\documentclass[11pt,a4paper]{article}

\title{On perfect flag-rank metric codes}

\usepackage{authblk}
\usepackage{relsize}
\usepackage{cite}
\usepackage{color}
\usepackage{amsmath,amsthm,amssymb}
\usepackage[margin=2.5cm]{geometry}
\usepackage{colortbl}
\usepackage[dvipsnames]{xcolor}
\usepackage{hyperref, enumitem}
\hypersetup{
  colorlinks   = true, %Colours links instead of ugly boxes
  urlcolor     = blue, %Colour for external hyperlinks
  linkcolor    = Purple, %Colour of internal links%
  citecolor   = red %Colour of citations
}

\usepackage{float}

\theoremstyle{definition}
\newtheorem{theorem}{Theorem}[section]
\newtheorem{definition}[theorem]{{{Definition}}}
\newtheorem{example}[theorem]{{{Example}}}

\newtheorem{remark}[theorem]{{{Remark}}}

\newtheorem{corollary}[theorem]{{{Corollary}}}%[theorem]
\newtheorem{proposition}[theorem]{{{Proposition}}}
\newtheorem{lemma}[theorem]{{{Lemma}}}

\newcommand{\numberset}{\mathbb}

\newcommand{\F}{\numberset{F}}
\newcommand{\Mat}{\mbox{Mat}}

\newcommand{\mC}{\mathcal{C}}

\newcommand{\Fq}{\F_q}

\newcommand{\rk}{\textnormal{rk}}

\DeclareMathOperator{\PG}{PG}

\author[1]{Gianira N. Alfarano}
\author[2]{Usman Mushrraf}
\author[2]{Ferdinando Zullo}
\affil[1]{Universit\'e de Rennes,  IRMAR, Rennes, France.\thanks{gianira-nicoletta.alfarano@univ-rennes.fr}}
\affil[2]{Dipartimento di Matematica e Fisica, Universit\`a degli Studi della Campania ``Luigi Vanvitelli'', I--\,81100 Caserta, Italy\thanks{\{usman.mushrraf,ferdinando.zullo\}@unicampania.it}}

\date{}

\begin{document}

\maketitle

\begin{abstract}
Flag-rank-metric codes arise as a natural generalization of rank-metric codes in the context of network communication. While recent research has mainly focused on algebraic and structural properties of these codes, the combinatorial geometry underlying the flag-rank metric remains largely unexplored. In this paper, we initiate a detailed investigation of this geometry. We explicitly determine the size of spheres of small flag-rank radius in the space $\mathrm{U}(n,\F_q)$ of upper triangular matrices over the finite field $\F_q$, and consequently obtain formulas for the size of balls of radius at most $3$. Using these enumerative results, we derive a sphere-packing bound for flag-rank-metric codes and introduce the notion of perfect codes with respect to the flag-rank metric. We observe that no non-trivial perfect flag-rank-metric codes exist in $\mathrm{U}(n,\F_q)$ for $n\in\{2,3\}$. We then investigate the possible parameters of perfect codes in higher dimensions. For minimum distance $3$, we obtain a characterization in terms of the codimension of the code, and show that suitable maximum flag-rank distance codes with minimum distance $3$ yield non-trivial perfect codes. For minimum distances $5$ and $7$, we derive explicit quadratic and cubic conditions, respectively, that any perfect code must satisfy. Finally, using asymptotic estimates for balls of fixed radius, we prove that for fixed length $n$ and $\delta\in\{3,5,7,9,11\}$, perfect linear flag-rank-metric codes with minimum distance $\delta$ do not exist over $\Fq$ for all sufficiently large $q$.
\end{abstract}

\medskip

\noindent \textbf{Keywords: } Flag-rank metric; perfect code; sphere-packing bound.

\noindent\textbf{MSC2020:} Primary 94B65; Secondary 94B05, 94B25, 94B60, 05B25.

\medskip

\section{Introduction}\label{sec:intro}

Network coding as a communication technique was introduced in the seminal work of Ahlswede, Cai, Li and Yeung~\cite{ahlswede2000network}, where intermediate nodes are allowed to combine incoming messages before forwarding them. This paradigm can substantially improve throughput, robustness, and latency in multicast networks. A major algebraic model for this framework was developed by K\"otter and Kschischang~\cite{koetter2008coding}, who proposed to encode information as vector subspaces of a finite-dimensional vector space. In this setting, a \emph{subspace code} is a set of subspaces of $\F^n$ equipped with the subspace distance, defined for every pair of subspaces $U,V\subseteq \F^n$ as
\[
\mathrm{d}_{\mathrm S}(U,V)=\dim(U+V)-\dim(U\cap V).
\]
This metric models errors and erasures in random network coding. Subspace codes, and in particular constant-dimension codes, have been extensively studied; see, for instance,~\cite{horlemann2018constructions, kurz2021constructions} and the references therein.

A fruitful approach to the study of constant-dimension codes arises from the representation of the Grassmannian $\mathrm{Gr}_k(\F^n)$ by reduced row echelon $k\times n$ matrices of rank $k$, modulo the natural action of $\mathrm{GL}(k,\F)$. This gives rise to a decomposition of the Grassmannian into \emph{Schubert cells}. The largest cell consists of subspaces $U_A\subseteq \F^n$ that admit a generator matrix of the form $(I_k\mid A)$, where $A\in \mathrm{Mat}(k,n-k,\F)$.
On this cell, the subspace distance is twice the rank distance, namely
\[
\mathrm{d}_{\mathrm S}(U_A,U_B)=2\rk(A-B).
\]
This observation creates a direct link between subspace coding and the theory of rank-metric codes, which has become a central topic in coding theory.

To overcome the restrictions imposed by constant-dimension codes, Liebhold, Nebe and V\'azquez-Castro proposed the use of \emph{flag codes}, whose codewords are chains of nested subspaces; see~\cite{liebhold2018network}. This idea has generated an active research area; see, for instance,~\cite{alonso2020flag,alonso2021optimum,alonso2021orbital,alonso2021cyclic,navarro2022flag,alonso2023flag,kurz2021bounds,stanojkovski2022submodule}. Algebraically, full flags live in the flag variety
\[
\mathrm{Fl}(\F^{n+1})
=
\{(U_1,\ldots,U_n)\mid U_1\subseteq\cdots\subseteq U_n,\ \dim(U_i)=i\},
\]
which admits a cell decomposition analogous to that of the Grassmannian.
A \emph{degenerate} version of the flag variety, denoted by $\mathrm{Fl}^{(a)}(\F^{n+1})$, was introduced from a geometric perspective in~\cite{feigin2012degeneration} and studied from a linear algebraic viewpoint in~\cite{irelli2015degenerate}. Its relevance to coding theory was recognized in~\cite{fourier2021degenerate}, where it was shown that the largest cell of $\mathrm{Fl}^{(a)}(\F^{n+1})$ is isometric to the space $\mathrm{U}(n,\F)$ of $n\times n$ upper triangular matrices over $\F$, endowed with the \emph{flag-rank metric}. This observation naturally leads to the notion of \emph{flag-rank-metric codes}, which may be regarded as the flag analogue of rank-metric codes.

The study of flag-rank-metric codes was initiated by Fourier and Nebe in~\cite{fourier2021degenerate}, where the first structural results and optimal constructions were obtained for specific parameter ranges. In~\cite{alfarano2024maximum}, a general Singleton-like bound relating the codimension and the minimum flag-rank distance of codes was established, extending the result of Fourier and Nebe. Motivated by this bound, the authors introduced and studied \emph{maximum flag-rank-distance codes}, or \emph{MFRD codes}, that is, codes whose parameters attain equality in the Singleton-like bound. They also provided new constructions and classifications for small distance values, as well as general constructions based on MRD codes, MDS codes, and finite-geometric techniques.

In this paper, we focus on the combinatorial aspects of the flag-rank metric over the finite field $\F_q$. More precisely, we study the sizes of spheres and balls in the metric space $\mathrm{U}(n,\F_q)$. We denote by $s_q(n,t)$ the size of the sphere of radius $t$ centered at the zero matrix $\boldsymbol{0}$ with respect to the flag-rank metric and by $b_q(n,t)$ the size of the corresponding ball. We explicitly determine $s_q(n,t)$ for $t=0,1,2,3$, and consequently obtain closed formulas for $b_q(n,t)$ for $t\le3$. 

Using these results, we derive a sphere-packing bound for flag-rank-metric codes. This allows us to introduce the notion of \emph{perfect codes} in the flag-rank metric. We first prove that no non-trivial perfect flag-rank-metric codes exist in $\mathrm{U}(n,\F_q)$ for $n\in\{2,3\}$. We then investigate the numerical conditions imposed by the existence of perfect codes for larger parameters, and show, in particular, that no perfect codes exist for even minimum distance $\delta$. For minimum distance $\delta=3$, we obtain an exact characterization in terms of the codimension of the code. This condition also shows that MFRD codes of minimum distance $3$ and length $n=q^2+q+1$, when they exist, are perfect.

We further analyze perfect codes with larger minimum distance. For $\delta=5$, we derive an explicit quadratic condition on the parameters, while for $\delta=7$ we obtain an explicit cubic condition. These equations provide strong arithmetic restrictions on the possible existence of perfect flag-rank-metric codes. In particular, they yield  obstructions and computational evidence excluding many possible parameter sets. We also obtain an asymptotic obstruction for perfect codes of fixed length. More precisely, for fixed $n$ and fixed odd minimum distance $\delta=2t+1$, we compare the asymptotic growth of the ball size with the codimension forced by the Singleton-like bound. This excludes perfect codes of fixed length for $\delta\in\{3,5,7,9,11\}$ and all sufficiently large~$q$.

The paper is organized as follows. In Section~\ref{sec:prel}, we recall the basic definitions and notation concerning flag-rank-metric codes. In Section~\ref{sec:balls}, we compute the sizes of spheres and balls of small radius in $\mathrm{U}(n,\F_q)$ and derive some general upper bounds. In Section~\ref{sec:sphere_packing}, we prove the sphere-packing bound and introduce perfect flag-rank-metric codes. Finally, in Section~\ref{sec:perfect_fr_codes}, we study perfect codes, proving non-existence results in small dimensions, deriving necessary numerical conditions for larger minimum distances, and establishing an asymptotic nonexistence result for fixed length and growing field size. We draw some conclusions in Section~\ref{sec:conclusions}.

\medskip
\section*{Acknowledgments}
This research was partially supported by Università Italo Francese (UIF/UFI) via PHC Galileo 2024 – G24-216. G. N. Alfarano is supported by the Agence Nationale de la Recherche through grant number ANR-24-CPJ1-0075-01. F. Zullo was partially supported by the Italian National Group for Algebraic and Geometric Structures and their Applications (GNSAGA - INdAM).
\medskip

\section{Preliminaries}\label{sec:prel}

In this section, we recall the basic definitions and notation concerning flag-rank-metric codes. For a detailed treatment, we refer the interested reader to \cite{liebhold2018network,alfarano2024maximum}.

Throughout the paper, we will use the following notation.

\medskip

\noindent{\textbf{Notation.}} 
Let $\F_q$ be the finite field with $q$ elements. We denote by
$\Mat(l,m,\F_q)$ the space of $l\times m$ matrices with entries in $\F_q$. Moreover, we denote by $\mathrm{U}(n,\F_q)$ the space of upper triangular matrices in $\Mat(n,n,\F_q)$, that is,
\[
\mathrm{U}(n,\F_q)=
\left\{
M=(m_{i,j})\in \Mat(n,n,\F_q) \;:\; m_{i,j}=0 \text{ whenever } i>j
\right\}.
\]
Let $M\in \mathrm{U}(n,\F_q)$ and let $1\le i\le n$. We denote by $M_{[i]}$ the $i\times(n-i+1)$ submatrix of $M$ obtained by deleting the first $i-1$ columns and the last $n-i$ rows. Finally, we denote by $\boldsymbol{0}$ the zero matrix.

\medskip 

We recall that $\mathrm{U}(n,\F_q)$ is an $\F_q$-vector space of dimension $\frac{n(n+1)}{2}$. We endow $\mathrm{U}(n,\F_q)$ with the \emph{flag-rank distance}, which is the metric induced by the \emph{flag-rank weight}.

\begin{definition}\label{def:frd}
The \textbf{flag-rank weight} of a matrix $M\in \mathrm{U}(n,\F_q)$ is defined as
\[
\mathrm{fr}(M):=\sum_{i=1}^{n}\rk(M_{[i]}).
\]
The \textbf{flag-rank distance} on $\mathrm{U}(n,\F_q)$ is the map
\[
\begin{array}{rccc}
\mathrm{d}_{\mathrm{fr}}:& \mathrm{U}(n,\F_q)\times \mathrm{U}(n,\F_q) & \longrightarrow & \mathbb N_0 \\
& (A,B) & \longmapsto & \mathrm{fr}(A-B).
\end{array}
\]
\end{definition}

\begin{example}
Let
\[
A=
\begin{pmatrix}
a & b\\
0 & c
\end{pmatrix}
\in \mathrm{U}(2,\F_q).
\]
Then
\[
\mathrm{fr}(A)=
\rk
\begin{pmatrix}
a & b
\end{pmatrix}
+
\rk
\begin{pmatrix}
b\\
c
\end{pmatrix}.
\]
Similarly, if
\[
B=
\begin{pmatrix}
a & b & c\\
0 & d & e\\
0 & 0 & f
\end{pmatrix}
\in \mathrm{U}(3,\F_q),
\]
then
\[
\mathrm{fr}(B)
=
\rk
\begin{pmatrix}
a & b & c
\end{pmatrix}
+
\rk
\begin{pmatrix}
b & c\\
d & e
\end{pmatrix}
+
\rk
\begin{pmatrix}
c\\
e\\
f
\end{pmatrix}.
\]
\end{example}

\begin{definition}
A \textbf{linear flag-rank-metric code} $\mC$ is an $\F_q$-subspace of $\mathrm{U}(n,\F_q)$ endowed with the flag-rank distance. The \textbf{minimum flag-rank distance} of $\mC$ is defined~as
\[
\mathrm{d}_{\mathrm{fr}}(\mC)
:=
\min\{\mathrm{fr}(M) : M\in \mC\setminus\{\boldsymbol{0}\}\}.
\]
If $\dim_{\F_q}(\mC)=k$ and $\mathrm{d}_{\mathrm{fr}}(\mC)=\delta$, then we say that $\mC$ is an $\{n,k,\delta\}_{\F_q}$ flag-rank-metric code.
\end{definition}

In \cite[Theorem 3.3]{alfarano2024maximum} a Singleton-like bound for flag-rank-metric codes was proved. We recall it in terms of the codimension of the code.

\begin{theorem}[Singleton-like bound]\label{th:Singleton_codimension}
Let $\mC\subseteq \mathrm{U}(n,\F_q)$ be a linear flag-rank-metric code with minimum distance $\delta$. Then
\begin{equation}\label{eq:sing-bound}
\mathrm{codim}_{\F_q}(\mC)
\ge
\delta-1
+
\left\lfloor\sqrt{\delta-1}\right\rfloor
\left\lfloor
\frac{\sqrt{4(\delta-1)+1}-1}{2}
\right\rfloor.
\end{equation}
\end{theorem}

\begin{definition}
A linear flag-rank-metric code is called a \textbf{maximum flag-rank-distance code}, or simply an \textbf{MFRD code}, if its parameters attain the Singleton-like bound of Eq.~\eqref{eq:sing-bound} with equality.
\end{definition}

The first examples of MFRD codes were obtained by Fourier and Nebe in
\cite{fourier2021degenerate}, who constructed MFRD codes for
$
\delta=\left\lceil \frac{n}{2} \right\rceil\cdot \left\lceil \frac{n+1}{2} \right\rceil
$
when $n$ is odd. In \cite{alfarano2024maximum}, this construction was
extended to the case where $n$ is even. Moreover, in \cite{alfarano2024maximum}, further families of
MFRD codes were obtained for
$
\delta=\left\lceil \frac{n}{2} \right\rceil\cdot \left\lceil \frac{n+1}{2} \right\rceil -1
$
when $n$ is odd. The same work also contains a classification of MFRD
codes with $\delta=2$, a characterization of MFRD codes with $\delta=3$
via the notion of \emph{support-avoiding codes} in the Hamming metric, and a
construction of MFRD codes with $\delta=4$ using auxiliary MDS codes in the Hamming metric and MRD codes in the rank metric. In particular,
MFRD codes are known to exist for several values of the minimum flag-rank distance, but a complete existence theory is still not available.

%%%%%%%%%%%%%%%%%%%%%%%%%%%%%%%%%%%%%%%%%%%%%%%%%%%%%%%%%%%%%%%%%%%%%%%%%%%%%%%%%

\section{On the size of balls and spheres in the flag-rank metric}\label{sec:balls}

We dedicate this section to the study of balls and spheres centered at matrices in $\mathrm{U}(n,\F_q)$, with respect to the flag-rank metric.
For a matrix $M\in \mathrm{U}(n,\F_q)$, we define the \textbf{flag-rank sphere of radius $t\in\mathbb N_0$ centered at~$M$} as
\[
S(M,t)=\{N\in \mathrm{U}(n,\F_q): \mathrm{d}_{\mathrm{fr}}(N,M)=t\},
\]
and the \textbf{flag-rank ball of radius $t\in\mathbb N_0$ centered at $M$} as
\[
B(M,t)=\{N\in \mathrm{U}(n,\F_q): \mathrm{d}_{\mathrm{fr}}(N,M)\le t\}.
\]
Clearly,
\[
B(M,t)=\bigcup_{i=0}^{t}S(M,i).
\]
Since $\mathrm{d}_{\mathrm{fr}}(N,M)=\mathrm{fr}(N-M)$, the size of balls and spheres does not depend on their center. Hence, throughout this section, we use the following notation:
\[
s_q(n,t):=|S(\boldsymbol{0},t)|
=
\left|\{A\in \mathrm{U}(n,\F_q):\mathrm{fr}(A)=t\}\right|,
\]
and
\[
b_q(n,t):=|B(\boldsymbol{0},t)|
=
\left|\{A\in \mathrm{U}(n,\F_q):\mathrm{fr}(A)\le t\}\right|.
\]
Thus
\[
b_q(n,t)=\sum_{i=0}^{t}s_q(n,i).
\]
In particular, determining the sizes of balls amounts to determining the values of $s_q(n,i)$.

We first observe that, for every $n\in\mathbb N$, the only matrix in $\mathrm{U}(n,\F_q)$ having flag-rank weight equal to zero is the zero matrix $\boldsymbol{0}$. Therefore, $s_q(n,0)=1$.

In the following, we determine the number of matrices of flag-rank weight $1$, $2$, and $3$.

\begin{proposition}
Let $M=(a_{i,j})\in \mathrm{U}(n,\F_q)$. Then $\mathrm{fr}(M)=1$ if and only if $M$ is diagonal and has exactly one nonzero diagonal entry. In particular, $s_q(n,1)=n(q-1)$.
\end{proposition}

\begin{proof}
Assume first that $\mathrm{fr}(M)=1$. Suppose that $a_{i,j}\neq0$ for some $i<j$. Then the entry $a_{i,j}$ appears in each of the submatrices $M_{[i]},M_{[i+1]},\ldots,M_{[j]}$.
In particular, both $M_{[i]}$ and $M_{[j]}$ are nonzero, and hence $\rk(M_{[i]})\ge1$, and $\rk(M_{[j]})\ge 1$.
This gives $\mathrm{fr}(M)\ge 2$, a contradiction. Therefore all entries of $M$ outside the main diagonal are zero.

Now suppose that two distinct diagonal entries, say $a_{i,i}$ and $a_{j,j}$ with $i\neq j$, are nonzero. Then $M_{[i]}$ and $M_{[j]}$ both have rank at least one, again implying $\mathrm{fr}(M)\ge2$. Hence exactly one diagonal entry of $M$ is nonzero.

Conversely, if $M$ is diagonal and has exactly one nonzero diagonal entry, then exactly one of the submatrices $M_{[i]}$ has rank one, while all the others have rank zero. Thus $\mathrm{fr}(M)=1$.

There are $n$ choices for the position of the nonzero diagonal entry and $q-1$ choices for its value. Therefore, $s_q(n,1)=n(q-1)$.
\end{proof}

\begin{proposition}\label{fr2z;general}
Let $n\ge2$. Then
\[
s_q(n,2)
=
(n-1)q^2(q-1)+\binom{n}{2}(q-1)^2.
\]
\end{proposition}

\begin{proof}
Let $M\in \mathrm{U}(n,\F_q)$ with $\mathrm{fr}(M)=2$. Since the flag-rank weight is the sum of the ranks of the submatrices $M_{[i]}$, exactly two of these submatrices have rank one and all the others have rank zero. We distinguish two cases.

First, suppose that all the entries of $M$ outside the main diagonal are zero. Then $\mathrm{fr}(M)=2$ if and only if exactly two diagonal entries are nonzero. This gives
\[
\binom{n}{2}(q-1)^2
\]
matrices.

Now suppose that $M$ has a nonzero entry outside the main diagonal. If $m_{i,j}\neq0$ with $j-i\ge2$, then this entry belongs to at least three submatrices, namely $M_{[i]},M_{[i+1]},\ldots,M_{[j]}$.
Hence $\mathrm{fr}(M)\ge3$, a contradiction. Thus the only possible nonzero entries outside the main diagonal lie on the first superdiagonal.
Let $m_{i,i+1}\neq0$ for some $i\in\{1,\ldots,n-1\}$. This entry belongs exactly to the two submatrices $M_{[i]}$ and $M_{[i+1]}$. Since $\mathrm{fr}(M)=2$, these must be the only nonzero submatrices, and both must have rank one. Therefore all entries of $M$ are zero except possibly $m_{i,i}, m_{i,i+1}, m_{i+1,i+1}$,
with $m_{i,i+1}\neq0$. For each fixed $i$, there are $q^2(q-1)$ such matrices. Since there are $n-1$ choices for $i$, this case contributes
\[
(n-1)q^2(q-1)
\]
matrices. The two cases are disjoint, and the formula follows by summing them.
\end{proof}

\begin{proposition}\label{fr:3}
Let $n\ge3$. Then
\[
s_q(n,3)
=
(n-2)q^4(q-1)
+
(n-1)(n-2)q^2(q-1)^2
+
\binom{n}{3}(q-1)^3.
\]
\end{proposition}

\begin{proof}
Let $M\in \mathrm{U}(n,\F_q)$ with $\mathrm{fr}(M)=3$. We classify the possible matrices according to the highest nonzero diagonal of $M$; see Figure~\ref{fig:fr3-configurations} for a visual explanation.
\begin{enumerate}
    \item [(a)] \emph{Only the main diagonal is involved.}
Suppose first that all entries outside the main diagonal of $M$ are zero. Then $\mathrm{fr}(M)=3$ if and only if exactly three diagonal entries are nonzero. This gives
\[
\binom{n}{3}(q-1)^3
\]
matrices.

\item [(b)] \emph{The first superdiagonal is involved, but the second superdiagonal is not.}
Assume that $m_{i,i+1}\neq0$ for some $i\in\{1,\ldots,n-1\}$ and that all entries on higher superdiagonals are zero. The entry $m_{i,i+1}$ belongs exactly to $M_{[i]}$ and $M_{[i+1]}$.
No other first-superdiagonal entry can be nonzero. Indeed, if an adjacent entry, say $m_{i+1,i+2}$, were nonzero, then $M_{[i+1]}$ would contain two independent nonzero positions, forcing $\rk(M_{[i+1]})\ge2$. Together with the nonzero blocks $M_{[i]}$ and $M_{[i+2]}$, this would imply $\mathrm{fr}(M)\ge4$. If a non-adjacent first-superdiagonal entry were nonzero, then at least four submatrices $M_{[j]}$ would be nonzero, again giving $\mathrm{fr}(M)\ge4$.
Thus there is exactly one nonzero entry on the first superdiagonal. 

To obtain flag-rank weight $3$, we must add exactly one nonzero diagonal entry outside the positions $i$ and $i+1$. The entries $m_{i,i}$ and $m_{i+1,i+1}$ may be chosen arbitrarily, since they do not increase the ranks of $M_{[i]}$ and $M_{[i+1]}$ beyond one.

Hence, for each fixed pair $(i,j)$, with $i\in\{1,\ldots,n-1\}$ and $j\in\{1,\ldots,n\}\setminus\{i,i+1\}$,
there are $q^2(q-1)^2$ matrices. Therefore this case contributes
\[
(n-1)(n-2)q^2(q-1)^2
\]
matrices.

\item[(c)] \emph{The second superdiagonal is involved.}
Assume that $m_{i,i+2}\neq0$ for some $i\in\{1,\ldots,n-2\}$. This entry belongs exactly to the three submatrices $M_{[i]}, M_{[i+1]}, M_{[i+2]}$.
Since $\mathrm{fr}(M)=3$, these must be the only nonzero submatrices, and each of them must have rank one.
Therefore all entries outside the rows and columns indexed by $i,i+1,i+2$ must be zero. Equivalently, the only entries that may be
nonzero are those in the upper triangular $3\times 3$ block supported
on the consecutive indices $i,i+1,i+2$, namely
$$
\begin{pmatrix}
m_{i,i} & m_{i,i+1} & m_{i,i+2}\\
0 & m_{i+1,i+1} & m_{i+1,i+2}\\
0 & 0 & m_{i+2,i+2}
\end{pmatrix},
$$
with $m_{i,i+2}\neq0$.

The condition that $M_{[i+1]}$ has rank one is equivalent to requiring
\[
(m_{i+1,i+1},m_{i+1,i+2})
=
\lambda(m_{i,i+1},m_{i,i+2})
\]
for some $\lambda\in\F_q$. Thus, for each fixed $i$, we may choose $m_{i,i+2}\in\F_q^*$, and $m_{i,i}$, $m_{i,i+1}$, $m_{i+2,i+2}$, $\lambda\in\F_q$.
This gives $q^4(q-1)$ matrices for each $i$. Since there are $n-2$ choices for $i$, this case contributes
\[
(n-2)q^4(q-1)
\]
matrices.

\end{enumerate}

\noindent
The three cases are disjoint and exhaustive, so the formula follows.
\end{proof}

\begin{figure}[H]
\centering
\small

\[
\begin{array}{c@{\qquad}c@{\qquad}c}
\textnormal{(a)} & \textnormal{(b)} & \textnormal{(c)} \\[4pt]
\begin{pmatrix}
\ast &        &        &        &        \\
0     & \ast  &        &        &        \\
0     & 0     & \ast  &        &        \\
0     & 0     & 0     & 0      &        \\
0     & 0     & 0     & 0      & 0
\end{pmatrix}
&
\begin{pmatrix}
\circ & \ast  &        &        &        \\
0     & \circ &        &        &        \\
0     & 0     & 0      &        &        \\
0     & 0     & 0      & \ast  &        \\
0     & 0     & 0      & 0      & 0
\end{pmatrix}
&
\begin{pmatrix}
x     & a         & b         &        &        \\
0     & \lambda a & \lambda b &        &        \\
0     & 0         & y         &        &        \\
0     & 0         & 0         & 0      &        \\
0     & 0         & 0         & 0      & 0
\end{pmatrix}
\\[6pt]
\begin{array}{c}
\text{three nonzero}\\
\text{diagonal entries}
\end{array}
&
\begin{array}{c}
\text{one nonzero first-}\\
\text{superdiagonal entry}\\
\text{and one isolated}\\
\text{nonzero diagonal entry}
\end{array}
&
\begin{array}{c}
\text{one nonzero second-}\\
\text{superdiagonal entry,}\\
\text{with } b\ne0,\ x,a,y,\lambda\in\F_q
\end{array}
\end{array}
\]

\caption{The three configurations contributing to $s_q(n,3)$. Here \(\ast\) denotes a nonzero entry, \(\circ\) denotes an arbitrary element of \(\F_q\), and blank entries are zero.}
\label{fig:fr3-configurations}
\end{figure}

We then get the following immediate size for the flag-rank balls of radius at most $3$.

\begin{corollary}\label{cor:ballsize}
For $n\ge3$, the sizes of the balls of flag-rank radius $0$, $1$, $2$, and $3$ in $\mathrm{U}(n,\F_q)$ are given by
\[
b_q(n,0)=1,
\]
\[
b_q(n,1)=1+n(q-1),
\]
\[
b_q(n,2)
=
1+n(q-1)
+
(n-1)q^2(q-1)
+
\binom{n}{2}(q-1)^2,
\]
and
\[
\begin{aligned}
b_q(n,3)
={}&
1+n(q-1)
+
(n-1)q^2(q-1)
+
\binom{n}{2}(q-1)^2  \\
&+
(n-2)q^4(q-1)
+
(n-1)(n-2)q^2(q-1)^2
+
\binom{n}{3}(q-1)^3.
\end{aligned}
\]
\end{corollary}

We now derive some general bounds on $s_q(n,t)$. The key observation is that matrices of small flag-rank weight cannot have nonzero entries too far from the main diagonal. We shall use the following result.

 \begin{proposition}{\cite[Lemma~4.12]{alfarano2024maximum}}\label{pro:nonzero}
Let $M\in \mathrm{U}(n,\F_q)$ and let $\delta=\mathrm{fr}(M)$. If $\delta<n$, then all nonzero entries of $M$ lie on the first $\delta$ diagonals.
\end{proposition}

\begin{proposition}\label{upper}
Let $\delta$ be an integer with $1\le \delta\le n$. Then
\[
s_q(n,\delta)
\le
q^{\frac{\delta(2n-\delta+1)}{2}}.
\]
\end{proposition}

\begin{proof}
If $\delta<n$, then Proposition~\ref{pro:nonzero} implies that every matrix $M\in \mathrm{U}(n,\F_q)$ with $\mathrm{fr}(M)=\delta$ is zero on all diagonals strictly above the $\delta$-th diagonal. Hence such a matrix is determined by its entries on the first $\delta$ diagonals.
The total number of entries on the first $\delta$ diagonals is
\[
N_\delta
=
n+(n-1)+\cdots+(n-\delta+1)
=
\frac{\delta(2n-\delta+1)}{2}.
\]
Therefore,
\[
s_q(n,\delta)\le q^{N_\delta}
=
q^{\frac{\delta(2n-\delta+1)}{2}}.
\]

\end{proof}
Note that if $\delta=n$, this bound is trivial, since the right-hand side is $q^{\frac{n(n+1)}{2}}=|\mathrm{U}(n,\F_q)|$.

As a consequence of Proposition~\ref{pro:nonzero}, we obtain the following lower bound on the number of zero entries of a matrix of given flag-rank weight.

\begin{corollary}\label{col:zero}
Let $M\in \mathrm{U}(n,\F_q)$ and suppose that $\mathrm{fr}(M)=\delta< n$. Then $M$ has at least
\[
\sum_{i=1}^{n-\delta}i
=
\frac{(n-\delta)(n-\delta+1)}{2}
\]
zero entries.
\end{corollary}

\begin{proof}
  By Proposition~\ref{pro:nonzero}, all entries strictly above the $\delta$-th diagonal are zero. The number of such entries is
\[
1+2+\cdots+(n-\delta)
=
\frac{(n-\delta)(n-\delta+1)}{2}.
\]
\end{proof}

We next obtain another simple restriction on the number of zero entries. We first have the following elementary fact.

\begin{lemma}\label{atleastnentries}
Let $M\in \mathrm{U}(n,\F_q)$. For every $i\in\{1,\ldots,n\}$, the submatrix $M_{[i]}$ contains at least $n$ entries.
\end{lemma}

\begin{proof}
By definition, $M_{[i]}$ has size $i\times(n-i+1)$, and therefore it contains $i(n-i+1)$ entries. Moreover, $i(n-i+1)-n =(i-1)(n-i)
\ge0$ for every $1\le i\le n$. Hence $i(n-i+1)\ge n$.
\end{proof}

\begin{lemma}\label{onezeroblock}
Let $M\in \mathrm{U}(n,\F_q)$. If $\mathrm{fr}(M)<n$, then there exists $i\in\{1,\ldots,n\}$ such that $M_{[i]}=\boldsymbol{0}$.
\end{lemma}

\begin{proof}
Suppose, by contradiction, that $M_{[i]}\neq\boldsymbol{0}$ for every $i\in\{1,\ldots,n\}$. Then $\rk(M_{[i]})\ge1$
for every $i$, and hence
\[
\mathrm{fr}(M)
=
\sum_{i=1}^{n}\rk(M_{[i]})
\ge n,
\]
contrary to the assumption.
\end{proof}

\begin{proposition}
Let $M\in \mathrm{U}(n,\F_q)$. If $\mathrm{fr}(M)<n$, then $M$ has at least $n$ zero entries.
\end{proposition}

\begin{proof}
By Lemma~\ref{onezeroblock}, there exists $i\in\{1,\ldots,n\}$ such that $M_{[i]}=\boldsymbol{0}$. By Lemma~\ref{atleastnentries}, the submatrix $M_{[i]}$ contains at least $n$ entries of $M$. Therefore $M$ has at least $n$ zero entries.
\end{proof}

We give a bound on the number of nonzero entries on each diagonal. 

\begin{proposition}\label{upperbound}
Let $M\in \mathrm{U}(n,\F_q)$ with $\mathrm{fr}(M)=\delta\leq n$. Then, for every $i\in\{1,\ldots,\delta\}$, the $i$-th diagonal of $M$ contains at most $\delta-i+1$ nonzero entries.
\end{proposition}

\begin{proof}
Fix $i\in\{1,\ldots,\delta\}$. Suppose that the $i$-th diagonal contains at least $\delta-i+2$ nonzero entries. Then there exist distinct indices $j_1<j_2<\cdots<j_{\delta-i+2}$
such that $m_{j_t,j_t+i-1}\neq0$, for every $t=1,\ldots,\delta-i+2$.
The entry $m_{j_t,j_t+i-1}$ belongs to each of the consecutive submatrices $M_{[j_t]},M_{[j_t+1]},\ldots,M_{[j_t+i-1]}$.
Since the starting indices $j_t$ are strictly increasing, the union of these intervals of indices has size at least $i+(\delta-i+2)-1=\delta+1$.
Thus at least $\delta+1$ of the submatrices $M_{[h]}$ are nonzero. Each of them has rank at least one, and so $\mathrm{fr}(M)\ge\delta+1$,
which contradicts the assumption $\mathrm{fr}(M)\leq\delta$.
\end{proof}

\begin{corollary}\label{cor:diagonal-bound}
Let $1\le \delta\le n$. Then
\[
s_q(n,\delta)
\le
\prod_{i=1}^{\delta}
\left(
\sum_{j=0}^{\delta-i+1}
\binom{n-i+1}{j}(q-1)^j
\right).
\]
\end{corollary}

\begin{proof}
Let $M\in \mathrm{U}(n,\F_q)$ be such that $\mathrm{fr}(M)=\delta$. By Proposition~\ref{upperbound}, for every $i\in\{1,\ldots,\delta\}$, the $i$-th diagonal of $M$ contains at most $\delta-i+1$ nonzero entries. Moreover, by Proposition~\ref{pro:nonzero}, there are no nonzero entries strictly above the $\delta$-th diagonal when $\delta<n$.

The $i$-th diagonal contains exactly $n-i+1$ entries. If exactly $j$ of them are nonzero, where $0\le j\le \delta-i+1$,
then their positions can be chosen in
\[
\binom{n-i+1}{j}
\]
ways, and their values can be chosen in $(q-1)^j$ ways. Hence the number of possible choices for the $i$-th diagonal is at most
\[
\sum_{j=0}^{\delta-i+1}
\binom{n-i+1}{j}(q-1)^j.
\]
Multiplying over $i=1,\ldots,\delta$ gives the desired bound.
\end{proof}

We conclude this section with an asymptotic upper bound for spheres of fixed
radius. The estimate follows directly from Corollary~\ref{cor:diagonal-bound}
and will be used in Section~\ref{sec:perfect_fr_codes}.

\begin{proposition}\label{prop:asymptotic-spheres-upper-bound}
Let $n$ and $t$ be fixed positive integers with $t\le n$. Then, as $q\to\infty$,
$$
s_q(n,t)=O_{n,t}\left(q^{t(t+1)/2}\right),
$$
were the notation $O_{n,t}$ means that the implicit constant depends only on $n$ and $t$.
Consequently,
$$
b_q(n,t)=O_{n,t}\left(q^{t(t+1)/2}\right).
$$
\end{proposition}

\begin{proof}
By Corollary~\ref{cor:diagonal-bound}, we have
$$
s_q(n,t)\le
\prod_{i=1}^t
\left(
\sum_{j=0}^{t-i+1}
\binom{n-i+1}{j}(q-1)^j
\right).
$$
For fixed $n$ and $t$, the $i$-th factor is a polynomial in $q$ of degree at
most $t-i+1$. Hence the whole product has degree at most
$$
\sum_{i=1}^t (t-i+1)=1+2+\cdots+t=\frac{t(t+1)}2.
$$
Therefore we get
$$
s_q(n,t)=O_{n,t}\left(q^{t(t+1)/2}\right).
$$
Since we have
$$
b_q(n,t)=\sum_{r=0}^t s_q(n,r),
$$
and the same estimate applied with $r$ in place of $t$ gives
$$
s_q(n,r)=O_{n,r}\left(q^{r(r+1)/2}\right)
=O_{n,t}\left(q^{t(t+1)/2}\right)
$$
for every $0\le r\le t$, we obtain
$$
b_q(n,t)=O_{n,t}\left(q^{t(t+1)/2}\right).
$$
\end{proof}

%%%%%%%%%%%%%%%%%%%%%%%%%%%%%%%%%%%%%%%%%%55

\section{Sphere-packing bound for flag-rank-metric codes}\label{sec:sphere_packing}

In this section, we prove a sphere-packing bound for the flag-rank metric. This bound naturally leads to the definition of perfect codes in this setting. As in the classical Hamming and rank-metric frameworks, the proof relies on the fact that balls of suitable radius centered at distinct codewords are pairwise disjoint.

\begin{theorem}[Sphere-packing bound]\label{thm:sphere-packing}
Let $\mC\subseteq \mathrm{U}(n,\F_q)$ be a flag-rank-metric code with minimum distance $\delta$, and let $t=\left\lfloor\frac{\delta-1}{2}\right\rfloor$.
Then
\[
|\mC|\, b_q(n,t)\le q^{\frac{n(n+1)}{2}}.
\]
\end{theorem}

\begin{proof}
Let $\mC=\{M_1,\ldots,M_s\}$.
We first show that the balls $B(M_i,t)$ are pairwise disjoint. Suppose, by contradiction, that there exist $i\neq j$ and
$N\in B(M_i,t)\cap B(M_j,t)$.
Then, by the triangle inequality,
\[
\mathrm{d}_{\mathrm{fr}}(M_i,M_j)
\le
\mathrm{d}_{\mathrm{fr}}(M_i,N)+\mathrm{d}_{\mathrm{fr}}(N,M_j)
\le
2t.
\]
Since $2t\le \delta-1$, this contradicts the fact that the minimum distance of $\mC$ is $\delta$. Hence the balls $B(M_i,t)$ are pairwise disjoint.
Therefore, we also have
\[
\left|\bigcup_{i=1}^{s} B(M_i,t)\right|
=
\sum_{i=1}^{s}|B(M_i,t)|.
\]
Since the size of a ball does not depend on its center, we have $|B(M_i,t)|=b_q(n,t)$ 
for every $i$. Hence
\[
\left|\bigcup_{i=1}^{s} B(M_i,t)\right|
=
s\,b_q(n,t)
=
|\mC|\,b_q(n,t).
\]
On the other hand, the union is contained in $\mathrm{U}(n,\F_q)$, and therefore
\[
|\mC|\,b_q(n,t)
\le
|\mathrm{U}(n,\F_q)|
=
q^{\frac{n(n+1)}{2}}.
\]
This proves the claim.
\end{proof}

\begin{remark}
The previous result holds for arbitrary, not necessarily linear, subsets $\mC\subseteq \mathrm{U}(n,\F_q)$. However, unless otherwise stated, we restrict our attention to linear flag-rank-metric codes, that is, $\F_q$-subspaces of $\mathrm{U}(n,\F_q)$.
\end{remark}

As in the classical theory of error-correcting codes, we define perfect codes as those attaining equality in the sphere-packing bound.

\begin{definition}
Let $\mC\subseteq \mathrm{U}(n,\F_q)$ be a flag-rank-metric code with minimum distance $\delta$, and let $t=\left\lfloor\frac{\delta-1}{2}\right\rfloor$.
We say that $\mC$ is \textbf{perfect} if
\[
|\mC|\,b_q(n,t)=q^{\frac{n(n+1)}{2}}.
\]
\end{definition}

\begin{example}
A trivial example of a perfect flag-rank-metric code is obtained by taking $\mC=\mathrm{U}(n,\F_q)$.
Indeed, in this case the minimum flag-rank distance is $\delta=1$, so that $t=\left\lfloor\frac{\delta-1}{2}\right\rfloor=0$ and $b_q(n,0)=1$. Hence
\[
|\mC|\,b_q(n,0)
=
|\mathrm{U}(n,\F_q)|
=
q^{\frac{n(n+1)}{2}}.
\]
\end{example}

\begin{proposition}\label{pro:perfect-odd-distance}
Let $\mathcal{C}\subseteq \mathrm{U}(n,\mathbb{F}_q)$ be a linear
flag-rank-metric code with minimum distance $\delta$. If $\mathcal{C}$ is
perfect, then $\delta$ is odd.
\end{proposition}

\begin{proof}
Suppose, by contradiction, that $\delta$ is even. Then $\delta=2e$ for some
integer $e\geq 1$. Hence the radius $t=\left\lfloor \frac{\delta-1}{2}\right\rfloor
 =\left\lfloor \frac{2e-1}{2}\right\rfloor
 =e-1$.
Let $C_0\in\mathcal{C}$ and choose $M\in \mathrm{U}(n,\mathbb{F}_q)$ such that $\mathrm{fr}(M)=e$.
Set $X=C_0+M$. Then, $
\mathrm{d}_{\mathrm{fr}}(X,C_0)=\operatorname{fr}(M)=e$,
and therefore $X\notin B_{e-1}(C_0)$.
Since $\mathcal{C}$ is perfect, the balls of radius $e-1$ centered at the
codewords of $\mathcal{C}$ cover the whole space. Hence there exists
$C_1\in\mathcal{C}$, with $C_1\neq C_0$, such that $X\in B_{e-1}(C_1)$.
Thus, $\mathrm{d}_{\mathrm{fr}}(X,C_1)\leq e-1.$
By the triangle inequality, we obtain
\[
\mathrm{d}_{\mathrm{fr}}(C_0,C_1)
\leq \mathrm{d}_{\mathrm{fr}}(C_0,X)+\mathrm{d}_{\mathrm{fr}}(X,C_1)
\leq e+(e-1)=2e-1.
\]
On the other hand, since $C_0$ and $C_1$ are distinct codewords and
$\mathcal{C}$ has minimum distance $\delta=2e$, we must have $\mathrm{d}_{\mathrm{fr}}(C_0,C_1)\geq \delta=2e$.
This is a contradiction. Therefore, $\delta$ cannot be even, and so $\delta$ is odd.
\end{proof}

\section{On perfect flag-rank-metric codes}\label{sec:perfect_fr_codes}

In this section, we study perfect flag-rank-metric codes. We first prove a necessary divisibility condition for the size of the ball appearing in the sphere-packing bound. We then use this condition to prove the non-existence of non-trivial perfect flag-rank-metric codes in $\mathrm{U}(2,\F_q)$ and $\mathrm{U}(3,\F_q)$.

\begin{proposition}\label{prop:condmodq}
Let $\mC\subsetneq \mathrm{U}(n,\F_q)$ be a linear flag-rank-metric code with minimum distance $\delta$, and let $t=\left\lfloor\frac{\delta-1}{2}\right\rfloor$.
If $\mC$ is perfect, then $b_q(n,t)\equiv0\pmod q$.
\end{proposition}

\begin{proof}
Since $\mC$ is linear and perfect, equality holds in the sphere-packing bound. Hence
\begin{equation}\label{eq:condcor}
q^{\dim_{\F_q}(\mC)}\, b_q(n,t)
=
q^{\frac{n(n+1)}2}.
\end{equation}
Therefore
\[
b_q(n,t)
=
q^{\frac{n(n+1)}2-\dim_{\F_q}(\mC)}
=
q^{\mathrm{codim}_{\F_q}(\mC)}.
\]
Since $\mC\subsetneq \mathrm{U}(n,\F_q)$, its codimension is positive. Thus $b_q(n,t)$ is divisible by $q$.
\end{proof}

\subsection{Non-existence of perfect codes in $\mathrm{U}(2,\F_q)$ and $\mathrm{U}(3,\F_q)$}

Combining Proposition~\ref{prop:condmodq} with the explicit formula for $b_q(n,1)$ and with the Singleton-like bound, we obtain the following non-existence result.

\begin{corollary}\label{cor:nonexistence-n2-n3}
There are no non-trivial perfect flag-rank-metric codes in $\mathrm{U}(n,\F_q)$ for $n\in\{2,3\}$.
\end{corollary}

\begin{proof}
Let $\mC\subsetneq \mathrm{U}(n,\F_q)$ be a linear flag-rank-metric code with minimum distance $\delta$, and set $t=\left\lfloor\frac{\delta-1}{2}\right\rfloor$.
For $n=2$, the maximum possible flag-rank distance is $2$, while for $n=3$ it is $4$. Hence $t\le 1$. 
If $t=0$, then equality in the sphere-packing bound gives $|\mC|=|\mathrm{U}(n,\F_q)|$, which contradicts the assumption $\mC\subsetneq \mathrm{U}(n,\F_q)$. Therefore, if $\mC$ is perfect and non-trivial, we must have $t=1$.
By Proposition~\ref{prop:condmodq}, since $\mC$ is perfect then $b_q(n,1)\equiv0\pmod q$.
Moreover, from Corollary~\ref{cor:ballsize} $b_q(n,1)=1+n(q-1)$, we get $1+n(q-1)\equiv 1-n \pmod q$.
Thus $q$ must divide $n-1$. For $n=2$, this is impossible. For $n=3$, it is possible only when $q=2$.

It remains to consider the exceptional case $(n,q)=(3,2)$. In this case,
\[
b_q(3,1)=1+3(2-1)=4.
\]
If $\mC$ were perfect, then \eqref{eq:condcor} would imply
\[
|\mC|
=
\frac{|\mathrm{U}(3,\F_2)|}{b_2(3,1)}
=
\frac{2^6}{4}
=
2^4.
\]
Thus $\dim_{\F_2}(\mC)=4$. Since $t=1$, we have $\delta\in\{3,4\}$. The Singleton-like bound gives
$\mathrm{codim}_{\F_2}(\mC)\ge3$,
and hence $\dim_{\F_2}(\mC)\le 6-3=3$, yielding a contradiction.
\end{proof}

The previous corollary shows that the first ambient spaces are too small to contain non-trivial perfect flag-rank-metric codes. In the next results, we derive necessary numerical conditions for perfect codes with larger minimum distance.

\begin{proposition}\label{prop:dim-lower-perfect}
Let $\mC\subseteq \mathrm{U}(n,\F_q)$ be a perfect linear flag-rank-metric code with minimum distance $\delta\le n$, and let
$t=\left\lfloor\frac{\delta-1}{2}\right\rfloor$.
Then
\[
\dim_{\F_q}(\mC)
\ge
\frac{n(n+1)}2-\frac{t(2n-t+1)}2.
\]
In particular,
\[
\dim_{\F_q}(\mC)
\ge
\frac{n(n+1)}2-\frac{\delta(2n-\delta+1)}2.
\]
\end{proposition}

\begin{proof}
Since $\mC$ is perfect, we have
\[
|\mC|
=
\frac{|\mathrm{U}(n,\F_q)|}{b_q(n,t)}
=
\frac{q^{\frac{n(n+1)}2}}{b_q(n,t)}.
\]
Since $t<\delta\le n$, every matrix of flag-rank weight at most $t$ has all its nonzero entries on the first~$t$ diagonals. Thus by Proposition \ref{upper}, we have
\[
b_q(n,t)
\le
q^{\frac{t(2n-t+1)}2}.
\]
It follows that
\[
|\mC|
\ge
q^{\frac{n(n+1)}2-\frac{t(2n-t+1)}2}.
\]
Taking logarithms in base $q$ gives
\[
\dim_{\F_q}(\mC)
\ge
\frac{n(n+1)}2-\frac{t(2n-t+1)}2.
\]
\end{proof}

\subsection{Perfect codes with minimum distance $3$}

\begin{theorem}\label{Thm:perf}
Let $\mC\subseteq \mathrm{U}(n,\F_q)$ be a linear flag-rank-metric code with minimum distance $\delta=3$, and let $\alpha$ be its codimension. Then $\mC$ is perfect if and only if
\[
n=\frac{q^\alpha-1}{q-1}.
\]
Moreover, necessarily $\alpha\ge \delta$.
\end{theorem}

\begin{proof}
Since $\delta=3$, we have $t=\left\lfloor\frac{\delta-1}{2}\right\rfloor=1$.
Thus $\mC$ is perfect if and only if
\[
b_q(n,1)=\frac{|\mathrm{U}(n,\F_q)|}{|\mC|}.
\]
Since $b_q(n,1)=1+n(q-1)$
and $\frac{|\mathrm{U}(n,\F_q)|}{|\mC|}
=
q^{\mathrm{codim}_{\F_q}(\mC)}
=
q^\alpha$,
this is equivalent to $1+n(q-1)=q^\alpha$.
Solving the identity  for $n$, we obtain $n=\frac{q^\alpha-1}{q-1}$.
Finally, the inequality $\alpha\ge\delta$ follows from the Singleton-like bound in Theorem~\ref{th:Singleton_codimension}.
\end{proof}

\begin{corollary}\label{cor:MFRD-perfect-34}
Let $\mC\subseteq \mathrm{U}(n,\F_q)$ be an MFRD code with minimum distance $\delta=3$. If
\[
n=\frac{q^3-1}{q-1},
\]
then $\mC$ is perfect.
\end{corollary}

\begin{proof}
Let $\alpha$ be the codimension of $\mC$. Since $\mC$ is MFRD, it attains the
Singleton-like bound. For $\delta=3$ the bound gives $\alpha=3$. Hence, $\alpha=\delta$.
Therefore, if
\[
n=\frac{q^3-1}{q-1},
\]
then
\[
n=\frac{q^\alpha-1}{q-1}.
\]
The claim follows from Theorem~\ref{Thm:perf}.
\end{proof}

\begin{corollary}\label{cor:perf-34-congruence}
Let $\mC\subseteq \mathrm{U}(n,\F_q)$ be a perfect linear flag-rank-metric code with minimum
distance $\delta=3$. Then $n\equiv 1 \pmod q$.
In particular, if $n\not\equiv 1\pmod q$, then no perfect linear flag-rank-metric code with
minimum distance $\delta=3$ exists in $\mathrm{U}(n,\F_q)$.
\end{corollary}

\begin{proof}
Let $\alpha$ be the codimension of $\mC$. By Theorem~\ref{Thm:perf},
\[
n=\frac{q^\alpha-1}{q-1}=1+q+\cdots+q^{\alpha-1}.
\]
Reducing modulo $q$, we obtain $n\equiv 1\pmod q$.
\end{proof}

\medskip

\medskip
We now recall the construction of MFRD codes with minimum distance $3$ from
\cite[Construction C and Theorem 4.21]{alfarano2024maximum}, and we show that, for
a suitable choice of the length, it yields perfect flag-rank-metric codes.

Let
$$
D_2(n,\F_q)=\{A\in \mathrm{U}(n,\F_q): a_{i,j}=0 \text{ whenever } j-i\ge 2\}.
$$
We identify $D_2(n,\F_q)$ with $\F_q^{2n-1}$ via
$$
\rho\left(\sum_{i=1}^{n}\lambda_iE_i+\sum_{i=1}^{n-1}\mu_iF_i\right)
=
(\lambda_1,\mu_1,\lambda_2,\ldots,\mu_{n-1},\lambda_n),
$$
where $E_i$ and $F_i$ are supported in positions $(i,i)$ and $(i,i+1)$,
respectively. Denote by $\PG(2,q)$ the projective plane with underlying vector space $\F_q^3$.

Assume that $3\le n\le q^2+q+1$. Choose distinct points
$P_i=[u_i]\in\PG(2,q)$, $i\in[n]$, and points
$Q_i=[v_i]\notin\langle P_i,P_{i+1}\rangle$, $i\in[n-1]$. Let
$$
H=(u_1\mid v_1\mid u_2\mid v_2\mid\cdots\mid v_{n-1}\mid u_n)
\in \Mat(3,2n-1,\F_q),
$$
set $D=\ker(H)$, and define
$$
\mathcal{D}:=
\rho^{-1}(D)\oplus \overline{U}(n-2,\F_q)
\subseteq \mathrm{U}(n,\F_q),
$$
where $\overline{U}(n-2,\F_q)$ denotes the embedded copy of
$\mathrm{U}(n-2,\F_q)$ supported on the diagonals strictly above the first
superdiagonal.

\begin{lemma}{\cite[Theorem~4.20]{alfarano2024maximum}}\label{res:mfrd-d3}
The code $\mathcal{D}$ is an $\left\{n,\frac{n(n+1)}{2}-3,3\right\}_{\F_q}$ MFRD code. Furthermore, an
$\left\{n,\frac{n(n+1)}{2}-3,3\right\}_{\Fq}$ MFRD code exists if and only if $
n \leq |\mathrm{PG}(2,q)|=q^2+q+1$.
\end{lemma}

\begin{theorem}\label{thm:constructed-perfect-delta3}
Let $n=q^2+q+1$. Then the code $\mC$ from~\cite[Construction C]{alfarano2024maximum} is a non-trivial perfect linear
flag-rank-metric code with minimum distance $3$ and codimension $3$.
\end{theorem}

\begin{proof}
By Lemma \ref{res:mfrd-d3}, the code
$\mathcal{D}$ from ~\cite[Construction C]{alfarano2024maximum} is an MFRD code with minimum distance $3$ whenever $3\le n\le q^2+q+1$. Hence, for $n=q^2+q+1$, the code
$\mathcal{D}$ is an MFRD code with minimum distance $3$.
Since the Singleton-like bound gives codimension $3$ for minimum distance $3$, we have
$\mathrm{codim}_{\F_q}(\mathcal{D})=3$. Moreover,
$$
n=q^2+q+1=\frac{q^3-1}{q-1}.
$$
Therefore the numerical condition in Theorem~\ref{Thm:perf} is satisfied with
$\alpha=3$. Hence $\mathcal{D}$ is perfect.
\end{proof}

\subsection{Perfect codes with minimum distance $5$}

\begin{theorem}\label{Thm:t2-perfect}
Let $\mC\subseteq \mathrm{U}(n,\F_q)$ be a linear flag-rank-metric code with minimum distance $\delta=5$, and let $\alpha$ be its codimension. Then $\mC$ is perfect if and only if
\[
n=
\frac{-(2q^2-q+3)+\sqrt{(2q^2+q+1)^2+8(q^\alpha-q)}}{2(q-1)}.
\]
Moreover, necessarily $\alpha>\delta$.
\end{theorem}

\begin{proof}
Since $\delta=5$, we have $t=\left\lfloor\frac{\delta-1}{2}\right\rfloor=2$.
Thus $\mC$ is perfect if and only if
\[
b_q(n,2)=\frac{|\mathrm{U}(n,\F_q)|}{|\mC|}=q^\alpha.
\]
By Corollary~\ref{cor:ballsize},
\[
b_q(n,2)
=
1+n(q-1)+(n-1)q^2(q-1)+\binom{n}{2}(q-1)^2.
\]
Therefore being perfect is equivalent to
\[
1+n(q-1)+(n-1)q^2(q-1)+\binom{n}{2}(q-1)^2
=
q^\alpha.
\]
Expanding and collecting the terms in $n$, we obtain
\[
\frac{(q-1)^2}{2}n^2
+
\frac{(q-1)(2q^2-q+3)}{2}n
+
1-q^\alpha-q^3+q^2
=
0.
\]
Solving this quadratic equation gives
\[
n=
\frac{-(2q^2-q+3)\pm\sqrt{(2q^2+q+1)^2+8(q^\alpha-q)}}{2(q-1)}.
\]
Since $n$ is positive, only the positive sign is possible. Hence
\[
n=
\frac{-(2q^2-q+3)+\sqrt{(2q^2+q+1)^2+8(q^\alpha-q)}}{2(q-1)}.
\]
Finally, the condition $\alpha>\delta$ follows from the Singleton-like bound.
\end{proof}

\begin{corollary}
Let $p$ be an odd prime. If $\mC\subseteq \mathrm{U}(n,\F_p)$ is a perfect flag-rank-metric code with minimum distance $\delta=5$, then $n\equiv1$ or $n\equiv2
\pmod p$.
\end{corollary}

\begin{proof}
If $\mC$ is perfect, then by Theorem~\ref{Thm:t2-perfect},
\[
1+n(p-1)+(n-1)p^2(p-1)+\binom{n}{2}(p-1)^2
=
p^\alpha.
\]
Reducing modulo $p$, we obtain
\[
1-n+\binom{n}{2}\equiv0\pmod p.
\]
Since $p$ is odd, this is equivalent to
\[
2-2n+n(n-1)\equiv0\pmod p,
\]
that is,
\[
(n-1)(n-2)\equiv0\pmod p.
\]
Therefore, $n\equiv1$ or $n\equiv2
\pmod p$.
\end{proof}

\begin{remark}
For $p=2$, the same reduction has to be handled separately, since one cannot divide by $2$ modulo $2$. In this case the condition becomes
\[
1+n+\binom{n}{2}\equiv0\pmod 2,
\]
which is equivalent to $n\equiv1$ or $n\equiv2
\pmod 4$.
\end{remark}

\begin{proposition}
Assume that $q=2$ and let $\mC\subseteq \mathrm{U}(n,\F_2)$ be a linear flag-rank-metric code with minimum distance $\delta=5$ and codimension $\alpha$. If $\alpha$ is odd and $\alpha\ge11$, then $\mC$ is not perfect.
\end{proposition}

\begin{proof}
For $q=2$, Theorem~\ref{Thm:t2-perfect} gives
$n=\frac{-9+\sqrt{\Delta}}{2}$,
where $\Delta=105+2^{\alpha+3}$.
Since $\alpha$ is odd, $\alpha+3$ is even. Let $\zeta=2^{\frac{\alpha+3}{2}}$. Then, 
$\zeta^2=2^{\alpha+3}$. Clearly, $\zeta^2<\Delta$.
Moreover, since $\alpha\ge 11$, we have
$105<2^{\frac{\alpha+5}{2}}+1$,
and hence $\Delta
=
2^{\alpha+3}+105
<
2^{\alpha+3}+2^{\frac{\alpha+5}{2}}+1
=
(\zeta+1)^2$.
Therefore $\Delta$ lies strictly between the two consecutive squares $\zeta^2$ and $(\zeta+1)^2$. Hence $\Delta$ is not a perfect square, so $n$ is not an integer. Consequently, $\mC$ cannot be perfect.
\end{proof}

\begin{remark}
Let $\mC\subseteq \mathrm{U}(n,\F_3)$ be a flag-rank-metric code of codimension $\alpha=7$ and minimum distance $\delta=5$. If such a code satisfies the numerical condition of Theorem~\ref{Thm:t2-perfect}, then necessarily
\[
[n,k,\delta]=[29,428,5].
\]
Indeed, for $q=3$ and $\alpha=7$, the formula in Theorem~\ref{Thm:t2-perfect} gives $n=29$, and therefore
\[
k=\frac{29\cdot30}{2}-7=428.
\]
We also performed computations in \textsc{Magma} for all integers $\alpha$ and all prime powers $q$ satisfying
$6\le \alpha\le100$, and $2\le q\le97$.
In all cases, the value of $n$ given by Theorem~\ref{Thm:t2-perfect} was not an integer, except for the case $q=3$ and $\alpha=7$. Hence, within this range, no other candidate parameter set occurs.
\end{remark}

\subsection{Perfect codes with minimum distance $7$}

\begin{theorem}\label{Thm:t3-perfect}
Let $\mC\subseteq \mathrm{U}(n,\F_q)$ be a linear flag-rank-metric code with minimum distance $\delta=7$, and let $\alpha$ be its codimension. Then $\mC$ is perfect if and only if $n\ge3$ satisfies
\begin{equation}\label{eq:cubic-perfect}
\begin{aligned}
&(q-1)^3n^3
+3(q-1)^2(2q^2-q+2)n^2 \\
&\quad +(q-1)(6q^4-18q^3+26q^2-7q+11)n \\
&\quad +6\bigl(-2q^5+4q^4-5q^3+3q^2+1-q^\alpha\bigr)
=0.
\end{aligned}
\end{equation}
Moreover, necessarily $\alpha\ge10$.
\end{theorem}

\begin{proof}
Since $\delta=7$, we have $t=\left\lfloor\frac{\delta-1}{2}\right\rfloor=3$.
Thus $\mC$ is perfect if and only if
\[
b_q(n,3)=\frac{|\mathrm{U}(n,\F_q)|}{|\mC|}=q^\alpha.
\]
By Corollary~\ref{cor:ballsize},
\[
\begin{aligned}
b_q(n,3)
={}&
1+n(q-1)
+
(n-1)q^2(q-1)
+
\binom{n}{2}(q-1)^2\\
&+
(n-2)q^4(q-1)
+
(n-1)(n-2)q^2(q-1)^2
+
\binom{n}{3}(q-1)^3.
\end{aligned}
\]
Hence being perfect is equivalent to
\[
\begin{aligned}
&1+n(q-1)
+
(n-1)q^2(q-1)
+
\binom{n}{2}(q-1)^2\\
&\quad+
(n-2)q^4(q-1)
+
(n-1)(n-2)q^2(q-1)^2
+
\binom{n}{3}(q-1)^3
=
q^\alpha.
\end{aligned}
\]
Multiplying by $6$ and simplifying yields precisely Equation~\eqref{eq:cubic-perfect}.
Finally, the condition $\alpha\ge10$ follows from the Singleton-like bound. 
\end{proof}

\begin{remark}
Using \textsc{Magma}, we checked equation~\eqref{eq:cubic-perfect} for all integers $\alpha$ and all prime powers $q$ satisfying $10\le \alpha\le100$, and $2\le q\le97$.
In all cases, the corresponding value of $n$ was not an integer. Hence, within this range, Equation~\eqref{eq:cubic-perfect} gives no candidate values of $n$ for a perfect flag-rank-metric code with minimum distance $\delta=7$.
\end{remark}

\subsection{An asymptotic result}

We conclude with an asymptotic nonexistence result for perfect codes of fixed length and small distance values.

\begin{proposition}\label{prop:asymptotic-obstruction-perfect}
Fix $n$ and let $\delta=2t+1$ with $t<n$. Suppose that
$$
\delta-1+
\left\lfloor\sqrt{\delta-1}\right\rfloor
\left\lfloor
\frac{\sqrt{4(\delta-1)+1}-1}{2}
\right\rfloor
>
\frac{t(t+1)}2.
$$
Then perfect linear flag-rank-metric codes in $\mathrm{U}(n,\Fq)$ with minimum distance
$\delta$ do not exist for all sufficiently large prime powers $q$.
\end{proposition}

\begin{proof}
Suppose that there exists a perfect linear flag-rank-metric code
$\mC\subseteq \mathrm{U}(n,\Fq)$ with minimum distance $\delta=2t+1$, and let
$\alpha=\mathrm{codim}_{\Fq}(\mC)$. Since $\mC$ is perfect and has packing radius $t$, we
have
$$
|\mC|b_q(n,t)=q^{n(n+1)/2}.
$$
As $\mC$ is linear, this is equivalent to $b_q(n,t)=q^\alpha$.
On the other hand, by Proposition~\ref{prop:asymptotic-spheres-upper-bound},
we have
$$
b_q(n,t)=O_{n,t}\left(q^{t(t+1)/2}\right)
$$
as $q\to\infty$.
By the Singleton-like bound, applied to a code of minimum distance $\delta$,
we have
$$
\alpha\ge
\delta-1+
\left\lfloor\sqrt{\delta-1}\right\rfloor
\left\lfloor
\frac{\sqrt{4(\delta-1)+1}-1}{2}
\right\rfloor.
$$
By assumption, this lower bound is strictly larger than $t(t+1)/2$. Hence
there exists an integer $\beta>t(t+1)/2$ such that $\alpha\ge \beta$.
Therefore, $q^\alpha\ge q^\beta$.
Thus $\mC$ perfect implies
$$
q^\beta\le q^\alpha=b_q(n,t)
=O_{n,t}\left(q^{t(t+1)/2}\right),
$$
which is impossible as $q\to\infty$. Hence such perfect codes do not exist for
all sufficiently large prime powers $q$.
\end{proof}

\begin{corollary}\label{cor:asymptotic-obstruction-delta-up-to-eleven}
Fix $n$. For each
$\delta\in\{3,5,7,9,11\},
$
perfect linear flag-rank-metric codes in $\mathrm{U}(n,\Fq)$ with minimum distance
$\delta$ do not exist for all sufficiently large prime powers $q$.
\end{corollary}

\begin{proof}
Write $\delta=2t+1$. For
$\delta=3,5,7,9,11$,
we have respectively
$t=1,2,3,4,5$.
By Proposition~\ref{prop:asymptotic-spheres-upper-bound},
$$
b_q(n,t)=O_{n,t}\left(q^{t(t+1)/2}\right).
$$
If $\mC\subseteq \mathrm{U}(n,\Fq)$ is a perfect linear flag-rank-metric code with
minimum distance $\delta$ and codimension $\alpha$, then $b_q(n,t)=q^\alpha$.
On the other hand, the Singleton-like bound gives respectively $\alpha\ge 3$, $\alpha\ge 6$, $\alpha\ge 10$, $\alpha\ge 12$, and $\alpha\ge 16$. Moreover,
$\frac{t(t+1)}2\in \{1,3,6,10,15\}$
for $t=1,2,3,4,5$, respectively.
Thus, in each case, the exponent forced by the Singleton-like bound is strictly larger than the exponent in the asymptotic upper bound for $b_q(n,t)$. Hence
the equality
$b_q(n,t)=q^\alpha
$ is impossible for all sufficiently large $q$. Therefore no such perfect codes
exist for all sufficiently large prime powers $q$.
\end{proof}

\begin{remark}
The previous corollary does not contradict the perfect codes with minimum
distance $3$ constructed in Theorem~\ref{thm:constructed-perfect-delta3}.
Indeed, in that family one has $
n=q^2+q+1$,
so the length is not fixed as $q$ tends to infinity.
\end{remark}

\section{Conclusions}\label{sec:conclusions}

In this work, we initiated the study of the combinatorial structure induced by the flag-rank metric on the space $\mathrm{U}(n,\F_q)$ of upper triangular matrices over a finite field. We explicitly determined the sizes of spheres of flag-rank radius $0$, $1$, $2$, and $3$, and consequently obtained closed formulas for the sizes of balls of radius at most $3$. In particular, these results give a complete description of balls and spheres in $\mathrm{U}(n,\F_q)$ for $n\in\{2,3\}$. We further provide an asymptotic estimate for the sizes of balls of larger radius. 

Using these results, we derived a sphere-packing bound for flag-rank-metric codes and introduced the notion of perfect codes in this setting. We proved that no non-trivial perfect flag-rank-metric codes exist in $\mathrm{U}(n,\F_q)$ for $n\in\{2,3\}$. We then investigated necessary numerical conditions for the existence of perfect codes in higher dimensions. For minimum distance $\delta=3$, we obtained an exact characterization in terms of the codimension $\alpha$, namely
\[
n=\frac{q^\alpha-1}{q-1}.
\]
For minimum distances $\delta=5$ and $\delta=7$, we derived explicit quadratic and cubic conditions, respectively, that the parameters of a perfect code must satisfy. These conditions provide arithmetic obstructions to the existence of perfect flag-rank-metric codes and exclude many possible parameter sets. Finally, using the asymptotic upper bound for the sizes of balls of fixed radius, we show that, for fixed length $n$ and $\delta\in\{3,5,7,9,11\}$, perfect linear flag-rank-metric codes with minimum distance $\delta$ do not exist over $\Fq$ for all sufficiently large $q$.

The results presented here show that, even in small radius, the flag-rank metric exhibits a rich and rigid combinatorial behavior. Several natural questions remain open. In particular, it would be interesting to determine the sizes of balls and spheres of arbitrary radius, to refine the arithmetic obstructions for perfect codes, and to understand whether non-trivial perfect flag-rank-metric codes exist beyond the cases constructed or ruled out in this paper.

\bibliographystyle{abbrv}
\bibliography{Biblio.bib}

\end{document}